\documentclass[a4j]{article}
\usepackage[T1]{fontenc}
\usepackage{lmodern}
\usepackage{textcomp}
\usepackage{latexsym}
\usepackage{comment}
\usepackage{amsmath}
\usepackage{amssymb}
\usepackage{amsfonts}
\usepackage{mathrsfs}
\usepackage{amsthm}
\usepackage{listings}
\newtheorem{thm}{Theorem}[section]
\newtheorem{defi}{Definition}[section]
\newtheorem{lem}{Lemma }[section]
\newtheorem{cor}{Corollary }[section]
\newtheorem{rem}{Remark}[section]
\newtheorem{prop}{Proposition}[section]
\newcommand{\pois}{\mathrm{Poi}}
\newcommand{\expo}{\mathrm{Exp}}
\renewcommand{\d}{\mathrm{d}}
\title{Upper Bound of Bayesian Generalization Error in Non-Negative Matrix Factorization}
\author{Naoki HAYASHI, Sumio WATANABE}



\begin{document}
 \columnseprule=0.3mm
\maketitle
\begin{abstract}
Non-negative matrix factorization ( NMF ) is a new knowledge discovery method that is 
used for text mining, signal processing, bioinformatics, and consumer analysis. 
However, its basic property as a learning machine is not yet clarified, as it is not a regular statistical model,
resulting that theoretical optimization method of NMF has not yet established. 
In this paper, we study the real log canonical threshold of NMF and give an upper bound of the generalization error in Bayesian learning. 
The results show that the generalization error of the matrix factorization can be made smaller 
than regular statistical models if Bayesian learning is applied. 
\end{abstract}
\section{Introduction}
Recently,
non-negative matrix factorization ( NMF ) \cite{Paatero,Cemgil} has been applied to text mining \cite{Xu}, signal processing \cite{Lee,Xu2,Jia}, bioinformatics \cite{Kim}, and
consumer analysis \cite{Kohjima}. 
Experiments has shown that a new knowledge discovery method is derived by NMF, however,
its mathematical property as a learning machine is not yet clarified, since it is not a regular statistical model.
A statistical model is called regular if a function from a parameter to a probability
density function is one-to-one and if the likelihood function can be approximated by a Gaussian function.

It is proved that, if a statistical model
is regular and if a true distribution is realizable by a statistical model, then the generalization error 
 is asymptotically  equal to $d/(2n)$,
where $d$, $n$, and the generalization error are the dimension of the parameter, the sample size, and
the expected Kullback-Leibler divergence of the true distribution and the estimated learning machine, respectively. 
However, 
the statistical model used in NMF is not regular because the map from a parameter to a probability
density function is not injective. As a result, its generalization error is still unknown. 
If there is a theoretical value of generalization error,  we can confirm whether the numerical calculation result is correct or not. 

There are many non-regular statistical models in machine learning. For example,
a neural network, a reduced rank regression,  a normal mixture, a hidden Markov model, and Boltzmann machine are 
such examples. From the theoretical point of view, the generalization error of a non-regular learning machine in Bayesian learning was 
proved to be asymptotically equal to $\lambda/n$, where $\lambda$ is a real log canonical threshold  ( RLCT ) \cite{Watanabe1,Watanabe2}. 
The RLCTs for several learning machines, have been clarified. In fact, 
 a 3-layer neural network \cite{Watanabe2}, a reduced rank regression \cite{Aoyagi1},  a normal mixture \cite{Yamazaki1}, and a hidden Markov model \cite{Yamazaki2},
they are clarified by using resolution of singularities. A statistical model selection method using RLCT has also been proposed \cite{Drton}. 

In this paper, we theoretically derive the upper bound of the RLCT of NMF, by which we can derive the upper bound of the Bayesian generalization error of NMF. This is the first derivation of the theoretical generalization error in Bayesian NMF. 
This paper consists of five parts.
 In the second section, we describe the upper bound of the RLCT in NMF in the Main Theorem.
 In the third section, we mathematically prepare the proof of the Main Theorem. In the fourth section,
 we prove the Main Theorem. And in the fifth section, we describe a theoretical application of the Main Theorem to Bayesian learning.

 \section{Framework and Main Result} 
 
 In this section, we explain the framework of Bayesian learning and introduce the main result of this paper. 
 
 \subsection{Framework of Bayesian Learning} 
 
 First, we explain the general theory of Bayesian learning. 
 
Let $q(x)$ and $p(x|\theta)$ be probability density functions on a finite dimensional real Euclidean space, where $\theta$ is a parameter. 
In learning theory, $q(x)$ and $p(x|\theta)$ represent a true distribution and a learning machine with $\theta$ respectively. 
A probability density function on a set of parameters $\varphi(\theta)$ is called a prior. 
Let $X^n=(X_1,X_2,...,X_n)$ be a set of random variables that are independently subject to $q(x)$, where $n$ and 
$X^n$ are referred to as the sample size and training data. 
The posterior distribution of $w$ is defined by
\[
p(\theta|X^n)=\frac{1}{Z}\varphi(\theta)\prod_{i=1}^n p(X_i|\theta),
\]
where $Z$ is the normalizing constant that is determined by the condition $\int p(\theta|X^n) \d \theta=1$. 
The Bayesian predictive distribution is also defined by
\[
p(x|X^n)=\int p(x|\theta)p(\theta|X^n)\d \theta. 
\]
The generalization error $G_n$ is defined by the Kullback-Leibler divergence from the true distribution $q(x)$ and
the predictive one $p(x|X^n)$,
\[
G=\int q(x)\log\frac{q(x)}{p(x|X^n)}\d x.
\]
 Note that $G$ is a function of $X^n$ hence it is also a random variable. Its expected value 
overall training data $\mathbb{E}[G]$ is called the expected generalization error. 
Assume there exists a parameter $\theta_0$ that satisfies $q(x)=p(x|\theta_0)$. 
By the singular learning theory \cite{Watanabe1, Watanabe2}, it was proven that 
\[
\mathbb{E}[G] =\frac{\lambda}{n}+o \left(\frac{1}{n} \right)
\]
holds when $n$ tends to infinity,  even if the posterior distribution can not be approximated by
any normal distribution. The constant $\lambda$ is the RLCT which is
an important birational invariant in algebraic geometry. From the mathematical point of view, RLCT is
characterized by the following property. We defined a zeta function by
\begin{equation}
\label{zeta}
\zeta(z)=\int K(\theta)^z\varphi(\theta)\d\theta,
\end{equation}
where
\[
K(\theta)=\int q(x)\log\frac{q(x)}{p(x|\theta)}\d x.
\]
Then this 
is holomorphic in $\mathrm{Re}(z)>0$ 
which can be analytically continued to a unique meromorphic function on the entire complex plane. 
The poles of this extended function are all negative rational numbers. 
Let $(-\lambda)$ be the nearest pole to the origin; $\lambda$ is then equal to the RLCT. 
If $p(x|\theta)$ is regular then $\lambda=d/2$; however, it is not usually general. In this paper we show an upper bound
of the RLCT of the NMF. 

\subsection{Main Theorem}

Second, we introduce the main result of this paper. 
In the followings, $\theta=(X,Y)$ is a parameter and $x=W$ is an observed random variable.

Let $\mathrm{M}(M,N,C)$ be a set of $M\times N$ matrices whose elements are in $C$, where $C$ is a subset of $\mathbb{R}$.
Let $K$ be a compact subset of $\mathbb{R}_{\geqq 0}=\{x \in \mathbb{R} | x \geqq 0\}$ and let 
$K_0$ be a compact of subset $\mathbb{R}_{>0}=\{x \in \mathbb{R} | x > 0\}$.
We denote that $X \! \in \! \mathrm{M}(M,H,K)$, $Y \! \in \! \mathrm{M}(H,N,K)$ and 
$A \in \mathrm{M}(M,H_0,K_0)$, $B \in \mathrm{M}(H_0,N,K_0)$ are NMFs of $AB$ such that they give the minimal $H_0$, where $H \geqq H_0$ and $\{(x,y,a,b) \in K^2 \times K_0^2|xy=ab \} \ne \emptyset$. We set that $\| \cdot \|$ is Frobenius norm.

\begin{defi}[{\bf An RLCT of NMF}]
Assume that the largest pole of the function of one complex variable $z$, 
\[
\zeta(z)=\int_{\mathrm{M}(M,H,K) } \mathrm{d}X \int_{\mathrm{M}(H,N,K) } \mathrm{d}Y
 \Bigl(\| XY-AB \|^2\Bigr)^z 
\]
is equal to $(-\lambda)$. Then 
$\lambda$ is said to be the RLCT of the NMF.
\end{defi}

In this paper, we prove the following theorem.

\begin{thm}[{\bf Main Theorem}]\label{thm:main}
The RLCT of NMF $\lambda$ satisfies the following inequality:
$$\lambda \leqq \frac{1}{2} \left[ (H-H_0)\min\{M,N\} +H_0 (M+N-1) \right].$$
If $H=H_0=1$, then the above bound gives the exact value.
\end{thm}

We prove this theorem in the next section. As an application of this theorem, we obtain an upper bound of the Bayesian generalization error of NMF. Sometimes NMF is studied in the case when only one target matrix is
decomposed, however, in general, decomposition of a set of independent matrices should be studied
because target matrices are often obtained daily, monthly, or different places \cite{Kohjima}. In such cases, 
decomposition of a set of matrices results in statistical inference. The following theorem shows 
a statistical bound of Bayesian estimation of NMF. 

\begin{thm}\label{thm:bayes}
Let the probability density functions of $W\in \mathrm{M}(M,N,K)$ be $q(W)$ and $p(W|X,Y)$, which represent 
a true distribution and a learning machine respectively defined by 
\begin{eqnarray*}
q(W) & \propto & \exp \left( -\frac{1}{2}\|W-AB\|^2 \right), \\
p(W|X,Y) & \propto & \exp \left(-\frac{1}{2}\|W-XY\|^2 \right). 
\end{eqnarray*}
Also let $\varphi(X,Y) $ be a probability density function such that it is positive on a compact subset of 
${\mathrm{M}(M,H,K) \times \mathrm{M}(H,N,K) }$ including $0$.
Then, the expected generalization error $\mathbb{E}[G]$ satisfies the following inequality:
\[
\mathbb{E}[G] \leqq \frac{1}{2n} [(H-H_0)\min\{M,N\} +H_0 (M+N-1)]
+o\left(\frac{1}{n}\right).
\]
If $H=H_0=1$, then the above bound gives the exact value.
\end{thm}

In this theorem, we study a case when a set of random matrices $W_1,W_2,...,W_n$ 
are observed and the true decomposition $A$ and $B$ are statistically estimated. A statistical model
$p(W|X,Y)$ which has parameters $(X,Y)$ are employed for estimation. Then the generalization error 
of Bayesian estimation is given by this theorem. 
If Theorem\ref{thm:main} is proved, Theorem\ref{thm:bayes} is immediately derived. 
Therefore, we prove Main Theorem\ref{thm:main} in the followings.

\section{Preparation}
We need the following five lemmas in order to prove the Main Theorem.
\begin{lem}
\label{lemH0}
Put $\theta = (X,Y)$, $\varphi=1$, and $K(X,Y)=\|XY\|^2$ in equlity $(\ref{zeta})$,where $X \in \mathrm{M}(M,H,K)$ and $Y \in \mathrm{M}(H,N,K)$. 
The RLCT $\lambda$ satisfies the following equality $:$
$$\lambda = \frac{H\min\{M,N\}}{2}.$$
\end{lem}
\begin{proof}
We set $X=(x_{ik})_{i,k=1,1}^{M,H},Y=(y_{kj})_{k,j=1,1}^{H,N}$, then
$$\|XY \|^2=\sum_{i=1}^M \sum_{j=1}^N \Biggl( \sum_{k=1}^H x_{ik}y_{kj} \Biggr)^2.$$
Then, since elements of matrices is nonnegative
$$\sum_{k=1}^H x_{ik}^2 y_{kj}^2 \leqq \Biggl( \sum_{k=1}^H x_{ik}y_{kj} \Biggr)^2 \leqq H \sum_{k=1}^H x_{ik}^2 y_{kj}^2.$$
As a log canonical threshold is not changed by any constant factor and it is order isomorphic, all we have to do is calculating an RLCT of 
$$\sum_{i=1}^M \sum_{j=1}^N \sum_{k=1}^H x_{ik}^2 y_{kj}^2 
=\sum_{k=1}^H \Biggl(\sum_{i=1}^M x_{ik}^2 \Biggr) \Biggl( \sum_{j=1}^N y_{kj}^2 \Biggr).$$

The RLCT $\lambda$ becomes a sum of each ones about $k$. For each $k$, 
we consider blow-ups of variables $\{x_{ik}\}$, $\{y_{kj}\}$ each other.
\begin{eqnarray*}
\lambda &=& \sum_{k=1}^H 
\min\left\{ \Biggl(\sum_{i=1}^M \frac{1}{2} \Biggr), \Biggl( \sum_{j=1}^N \frac{1}{2} \Biggr)
\right\} \\
&=& \sum_{k=1}^H \frac{\min\{M,N\}}{2} = \frac{H\min\{M,N\}}{2}.
\end{eqnarray*}
\end{proof}

\begin{rem}
Lemma $\ref{lemH0}$ means that the bound in Main Theorem formally gives the exact value if $H_0=0$ i.e. $AB=O$.
\end{rem}

\begin{lem}
\label{lemH1}
If $H_0=H=1$, the equal sign of the Main Theorem holds.
\end{lem}

We set 
$X \! = \! (x_{i})_{i=1}^M$. 
$Y \! = \! \left( (y_{j})_{j=1}^N \right)^T$, 
$A \! = \! (a_{i})_{i=1}^M$, 
and $B \! =\! \left( (b_{j})_{j=1}^N \right)^T$.

We prove the next lemma for the proof of Lemma \ref{lemH1}.
\begin{lem}
\label{lemineq}
We put $f_{kl}:=x_k y_l -a_k b_l \ (k,l\in \mathbb{N})$. Then $\forall M,N \in \mathbb{N}_{\geqq 2}=\{n\in\mathbb{N} | n \geqq 2\}$, $\exists C>0$ $s.t.$
\begin{eqnarray}
\label{induc}
f_{MN}^2 \leqq C \Biggl( \sum_{i=2}^M f_{i1}^2 + \sum_{j=2}^N f_{1j}^2 +f_{11}^2 \Biggr).
\end{eqnarray}
\end{lem}

\begin{proof}
We arbitrarily take $i,j(2 \leqq i \leqq M,2 \leqq j \leqq N)$ and fix them. We consider the case where $f_{ij}=0$. The elements of the matrices are positive, which causes
\begin{eqnarray*}
x_i y_j-a_i b_j &\!=& \frac{x_i}{x_{i-1}} (x_{i-1}y_j \!-a_{i-1}b_j) \\
&& +\frac{a_{i-1}b_j}{x_{i-1}y_{j-1}}(x_i y_{j-1} \!-a_i b_{j-1}) \\
&&-\frac{a_i b_j}{x_{i-1}y_{j-1}}(x_{i-1}y_{j-1} -a_{i-1}b_{j-1}).
\end{eqnarray*}
Generally, 
$$(\alpha + \beta +\gamma)^2 \leqq 2(\alpha^2 +\beta^2 +\gamma^2), \, \rm{for} \, \forall \alpha ,\beta ,\gamma \in \mathbb{R}$$
is attained. We square both side 
\begin{eqnarray*}
f_{ij}^2 & \leqq & 
2\Biggl(\! \Bigl(\frac{x_i}{x_{i-1}} \Bigr)^2 \! f_{(i-1)j}^2 
+ \Bigl(\frac{a_{i-1}b_j}{x_{i-1}y_{j-1}} \Bigr)^2 \! f_{i(j-1)}^2 
 +\Bigl(\frac{a_i b_j}{x_{i-1}y_{j-1}} \Bigr)^2 \! f_{(i-1)(j-1)}^2 \! \Biggr)
\end{eqnarray*}
Because of $X \in \mathrm{M}(M,1,K)$, $Y \in \mathrm{M}(1,N,K)$ and what $K$ is compact, 
$2\Bigl(\frac{x_i}{x_{i-1}} \Bigr)^2$
, $2\Bigl(\frac{a_{i-1}b_j}{x_{i-1}y_{j-1}} \Bigr)^2$ 
and $2\Bigl(\frac{a_i b_j}{x_{i-1}y_{j-1}} \Bigr)^2$ have a maximum value $m_{(i-1)j}$, $m_{i(j-1)}$ and $m_{(i-1)(j-1)}$ ,respectively. Thus 
$$f_{ij}^2 \leqq m_{(i-1)j} f_{(i-1)j}^2 +m_{i(j-1)} f_{i(j-1)}^2 +m_{(i-1)(j-1)} f_{(i-1)(j-1)}^2.$$
Besides, we set $C(i,j)=\max\{m_{(i-1)j},m_{i(j-1)},m_{(i-1)(j-1)} \}$ and get
\begin{eqnarray}
\label{element}
f_{ij}^2 \leqq C(i,j)(f_{(i-1)j}^2 + f_{i(j-1)}^2 + f_{(i-1)(j-1)}^2 ).
\end{eqnarray}
We prove inequality(\ref{induc}) by mathematical induction of $M,N \in \mathbb{N}_{\geqq 2}$
, using inequality(\ref{element}).

\noindent(Step 1) In the case of $(M,N) \in (\mathbb{N}_{\geqq 2})^2 \ \mathrm{s.t.} \ M+N=4,5$. 

$M=N=2 \Rightarrow$ clear. 
In case of $M=2 \wedge N=3$, using inequality(\ref{element}), we give
$$f_{23}^2 \leqq C(2,3)(f_{13}^2 + f_{22}^2 + f_{12}^2 )$$
and
$$f_{22}^2 \leqq C(2,2)(f_{12}^2 + f_{21}^2 + f_{11}^2 ).$$
We show the following inequality, using them.
$$f_{23}^2 \leqq 
C(2,3) \left\{ f_{13}^2 + C(2,2)( f_{12}^2 + f_{21}^2 + f_{11}^2 ) + f_{12}^2 \right\}.$$
Thus, if we set $C:=C(2,3)\max\{1,C(2,2)\}$, the inequality(\ref{induc}) is attained.
The case of $M=3 \wedge N=2$ can be proven in the same way.

\noindent(Step 2) We assume that inequality(\ref{induc}) is attained if $(M,N) \in (\mathbb{N}_{\geqq 2})^2 \ \mathrm{s.t.} \ M+N=i+j,i+j+1$.


(Case 2-1) In case of $i-k \geqq 2 \wedge j+k \geqq 2 \ \mathrm{i.e.} \ 2-j \leqq k \leqq i-2$.

Using inequality(\ref{element}), 
\begin{eqnarray*}
f_{(i-k+1)(j+k+1)}^2 \leqq C(i-k+1,j+k+1) \bigl( f_{(i-k)(j+k+1)}^2+ f_{(i-k+1)(j+k)}^2 + f_{(i-k)(j-k)}^2 \bigr).
\end{eqnarray*}
According to the above assumption, $\exists C_1,C_2,C_3>0$ s.t.
$$f_{(i-k)(j+k+1)}^2 \leqq C_1 \Biggl( \sum_{i=2}^{i-k} f_{i1}^2 + \sum_{j=2}^{j+k+1} f_{1j}^2 +f_{11}^2 \Biggr)$$
$$f_{(i-k+1)(j+k)}^2 \leqq C_2 \Biggl( \sum_{i=2}^{i-k+1} f_{i1}^2 + \sum_{j=2}^{j+k} f_{1j}^2 +f_{11}^2 \Biggr)$$
$$f_{(i-k)(j+k)}^2 \leqq C_3 \Biggl( \sum_{i=2}^{i-k} f_{i1}^2 + \sum_{j=2}^{j+k} f_{1j}^2 +f_{11}^2 \Biggr).$$
Using the above three inequalities, 
\begin{eqnarray*}
f_{(i-k+1)(j+k+1)}^2 &\leqq&\!C(i\!-\!k\!+\!1,j\!+\!k\!+\!1)\Biggl\{ \\
&& (C_1\!+\!C_2\!+\!C_3)\Biggl( \sum_{i=2}^{i-k} f_{i1}^2 + \sum_{j=2}^{j+k} f_{1j}^2 +f_{11}^2 \Biggr) \\
&& +C_1 f_{1j+k+1}^2 +C_2 f_{(i-k+1)1}^2
\Biggr\}.
\end{eqnarray*}
Hence we set $C:=\!C(i\!-\!k\!+\!1,j\!+\!k\!+\!1)(C_1\!+\!C_2\!+\!C_3)$, the inequality(\ref{induc}) is attained because of  $C_1,C_2,C_3>0$.

(Case 2-2) In case of $k=1-j$.

Using inequality(\ref{element}), 
$$f_{(i+j)2}^2 \leqq C(i+j,2)(f_{(i+j-1)2}^2 + f_{(i+j)1}^2 + f_{(i+j-1)1}^2).$$
According to the above assumption, $\exists C'>0$ s.t. 
$$f_{(i+j-1)(2)}^2 \leqq C' \Biggl( \sum_{i=2}^{i+j-1} f_{i1}^2 + f_{12}^2 +f_{11}^2 \Biggr).$$
Using the above inequality,  
$$f_{i+j2}^2 \leqq C(i+j,2)\Biggl\{ C' \Biggl( \sum_{i=2}^{i+j-1} f_{i1}^2 + f_{12}^2 +f_{11}^2 \Biggr)
+ f_{(i+j)1}^2 + f_{(i+j-1)1}^2 \Biggr\}.$$
Hence, we set $C:=C(i+j,2)\max\{1,C'\}$, and inequality(\ref{induc}) is attained.

(Case 2-3) In case of $k=i-1$

This can be derived in the same way as Case 2-2 and Lemma \ref{lemineq}, as follows. 
\end{proof}

According to Lemma \rm{\ref{lemineq}}, the next corollary is easily verified. 

\begin{cor}
\label{corineq}
Under the same assumption as Lemma \rm{\ref{lemineq}}, $\exists D>0$ s.t. 
$$ \sum_{i=2}^M f_{i1}^2 + \sum_{j=2}^N f_{1j}^2 +f_{11}^2 \leqq \sum_{i=1}^M \sum_{j=1}^N f_{ij}^2 \leqq D \Biggl( \sum_{i=2}^M f_{i1}^2 + \sum_{j=2}^N f_{1j}^2 +f_{11}^2 \Biggr).$$
\end{cor}

We immediately get the next lemma using Corollary \rm{\ref{corineq}}. 
\begin{lem}
\label{lemideal}
When an RLCT of $F$ is equal to one of $G$, it is denoted by $F \sim G$. 
Under the same assumption as Lemma \rm{\ref{lemineq}}, the following relation is attained:
\begin{eqnarray}
\label{ideal}
\sum_{i=1}^M \sum_{j=1}^N f_{ij}^2 \sim \sum_{i=2}^M f_{i1}^2 + \sum_{j=2}^N f_{1j}^2 +f_{11}^2.
\end{eqnarray}
\end{lem}

Using Lemma \rm{\ref{lemideal}, Lemma \ref{lemH1}} is proven. 
\begin{proof}[Proof of Lemma \ref{lemH1}]
We set $X=(x_{i})_{i=1}^M$, $Y=(y_{j})_{j=1}^N$, $A=(a_{i})_{i=1}^M$, $B=(b_{j})_{j=1}^N$, $f_{kl}=x_k y_l -a_k b_l , \ k,l\in \mathbb{N}$. 
On account of Lemma \rm{\ref{lemideal}}, 
$$\| XY-AB \|^2 = \sum_{i=1}^M \sum_{j=1}^N f_{ij}^2 \sim \sum_{i=2}^M f_{i1}^2 + \sum_{j=2}^N f_{1j}^2 +f_{11}^2.$$
Thus, all we have to do is calculate an RLCT of the right side. 
We arbitrarily take $i,j(1 \leqq i \leqq M, 1 \leqq j \leqq N, i,j \in \mathbb{N})$ and fix them.
Let $g:(x_i,y_j) \mapsto (\xi_i,f_{ij})$ be the following coordinate transformation:
\begin{eqnarray*}
  \xi_i &=& x_i \\
  f_{ij} &=& x_i y_j - a_i b_j .
\end{eqnarray*}
The Jacobi matrix of the above transformation is equal to
\[
  \frac{\partial(\xi_i ,f_{ij})}{\partial(x_i,y_j)}= \left(
    \begin{array}{cc}
  \frac{\partial \xi_i}{\partial x_i} & \frac{\partial f_{ij}}{\partial x_i} \\
  \frac{\partial \xi_i}{\partial y_j} & \frac{\partial f_{ij}}{\partial y_j}  \\
    \end{array}
  \right)
  = \left(
    \begin{array}{cc}
  1 & y_j \\
  0 & x_i  \\
    \end{array}
  \right).
\]
Because of
$$\Biggl| \frac{\partial(\xi_i ,f_{ij})}{\partial(x_i,y_j)} \Biggr|=x_i>0,$$
$g$ is an analytic isomorphism. Hence the zeta function of $(x_i y_j - a_i b_j)^2$ is equal to
\begin{eqnarray*}
  \iint_{K^2} (x_i y_j - a_i b_j)^{2z} \mathrm{d}x_i \mathrm{d}y_j 
  &=& \iint_{g(K^2)} f_{ij}^{2z} \xi_i \mathrm{d}\xi_i \mathrm{d}f_{ij}.\\
\end{eqnarray*}
Owing to $\xi_i  =  x_i >0$, $\xi_i$ does not contribute to the maximum pole of the zeta function. Therefore we can consider only 
$$\iint_{g(K^2)} f_{ij}^{2z} \mathrm{d}\xi_i \mathrm{d}f_{ij} .$$
Since $g$ is an analytic isomorphism and $K^2$ is compact, $\exists c_1,c_2,c_3,c_4 \in \mathbb{R} \, (c_1<c_2,c_3<c_4)$ s.t.
$$g(K^2)=\bigl\{(\xi_i,f_{ij}) \in \mathbb{R}_{>0}^2 | c_1 \leqq \xi_i \leqq c_2 \wedge c_3 \leqq f_{ij} \leqq c_4 \bigr\}$$
This is clearly integrable, therefore we calculate an iterated integral. Thus
\begin{eqnarray*}
  \iint_{g(K^2)} f_{ij}^{2z} \mathrm{d}\xi_i \mathrm{d}f_{ij}
  &=& \int_{c_1}^{c_2} \mathrm{d}\xi_i \int_{c_3}^{c_4} f_{ij}^{2z} \mathrm{d}f_{ij} \\ 
  &=& \frac{C}{2z+1} \ \rm{for} \ \exists C \ne 0(\rm{constant}).
\end{eqnarray*}
Therefore an RLCT of $f_{ij}^2=(x_i y_j - a_i b_j)^2$ is equal to 1/2.

According to Lemma \rm{\ref{lemideal}}, $f_{11},f_{21},...,f_{M1},f_{12},...,f_{1N}$ are independent and the number of them is $M+N-1$. 
An RLCT $\lambda$ of $\sum_{i=2}^M f_{i1}^2 + \sum_{j=2}^N f_{1j}^2 +f_{11}^2$ is equal to
$$\lambda =\sum_{i=2}^M \frac{1}{2} + \sum_{j=2}^N \frac{1}{2} +\frac{1}{2}= \frac{M+N-1}{2}. $$
\end{proof}

Let $X \! \in \! \mathrm{M}(M,H,K), Y \! \in \! \mathrm{M}(H,N,K)$ be
$$X\! = \!(x_1,\ldots,x_H) , x_{k}\!=\!(x_{ik})_{i=1}^{M},$$ 
$$Y\! = \!(y_1,\ldots,y_H)^T , y_{k}\!=\!(y_{kj})_{j=1}^{N},$$
and
$A \! \in \! \mathrm{M}(M,H_0,K_0), B \! \in \! \mathrm{M}(H_0,N,K_0)$ be
$$A\! =\! (a_1,\ldots,a_{H_0}) , a_{k}\!=\!(a_{ik})_{i=1}^{M},$$
$$B\! =\! (b_1,\ldots,b_{H_0})^T , b_{k}\!=\!(b_{kj})_{j=1}^{N},$$
respectively.
\begin{lem}
\label{lemH}
In the case of $H=H_0$, the Main Theorem is attained.
\end{lem}
\begin{proof}
\begin{eqnarray*}
&& \| XY-AB \|^2 \\
&=& \sum_{i=1}^M \sum_{j=1}^N (x_{i1}y_{1j} + ...+ x_{iH}y_{Hj} - a_{i1}b_{1j} - a_{iH}b_{Hj})^2 \\
&=& \sum_{i=1}^M \sum_{j=1}^N \Biggl( \sum_{k=1}^H (x_{ik}y_{kj} - a_{ik}b_{kj}) \Biggr)^2 \\
& \leqq & C \sum_{i=1}^M \sum_{j=1}^N \sum_{k=1}^H  (x_{ik}y_{kj} - a_{ik}b_{kj})^2 \ \rm{for} \ \exists C>0(\rm{const.}) \\
& \sim & \sum_{i=1}^M \sum_{j=1}^N \sum_{k=1}^H  (x_{ik}y_{kj} - a_{ik}b_{kj})^2 \\
&=& \sum_{k=1}^H \Biggl( \sum_{i=1}^M \sum_{j=1}^N (x_{ik}y_{kj} - a_{ik}b_{kj})^2 \Biggr) \\
&=& \sum_{k=1}^H 
\left\|
  x_k y_k^T - a_k b_k^T
\right\|^2.
\end{eqnarray*}
The variables in the $\sum$ of the last term are independent for $k$. That causes, using Lemma \rm{\ref{lemH1}}, 
an RLCT $\tilde{\lambda}$ of the term is equal to 
$$\tilde{\lambda} = \sum_{k=1}^H \left( \frac{M+N-1}{2} \right) = H\frac{M+N-1}{2}.$$
An RLCT is order isomorphic therefore we get
$$\lambda \leqq \tilde{\lambda} =H\frac{M+N-1}{2}.$$
\end{proof}
%
\section{Proof of Main Theorem}
\begin{proof}[Proof of Main Theorem]
\begin{eqnarray*}
  && \| XY-AB \|^2  \\
  &=& \sum_{i=1}^M \sum_{j=1}^N (x_{i1}y_{1j} + ... + x_{iH}y_{Hj} 
  - a_{i1}b_{1j} - ... - a_{iH_0}b_{H_0j})^2 \\
  &=& \sum_{i=1}^M \sum_{j=1}^N \Biggl( \sum_{k=1}^H x_{ik}y_{kj} - \sum_{k=1}^{H_0} a_{ik}b_{kj} \Biggr)^2 \\
  &=& \sum_{i=1}^M \sum_{j=1}^N \Biggl( \sum_{k=1}^{H_0} (x_{ik}y_{kj} - a_{ik}b_{kj}) 
+ \sum_{k=H_0 +1}^H x_{ik}y_{kj} \Biggr)^2 \\
  & \leqq & C \sum_{i=1}^M \sum_{j=1}^N 
\left( \sum_{k=1}^{H_0}  (x_{ik}y_{kj} - a_{ik}b_{kj})^2 + \sum_{k=H_0 +1}^H x_{ik}^2 y_{kj}^2 \right)  \ \rm{for} \ \exists C>0(\rm{const.}) \\
  & \sim & \sum_{i=1}^M \sum_{j=1}^N 
\left( \sum_{k=1}^{H_0}  (x_{ik}y_{kj} - a_{ik}b_{kj})^2 + \sum_{k=H_0 +1}^H x_{ik}^2 y_{kj}^2 \right) \\
  &=& \sum_{k=1}^{H_0} \Biggl( \sum_{i=1}^M \sum_{j=1}^N (x_{ik}y_{kj} - a_{ik}b_{kj})^2 \Biggr) 
+\sum_{k=H_0 +1}^H \Biggl( \sum_{i=1}^M \sum_{j=1}^Nx_{ik}^2 y_{kj}^2 \Biggr) \\
  &=& \sum_{k=1}^{H_0}
\left\|
 x_k y_k^T - a_k b_k^T
\right\|^2 
 + 
\left\|
  \left(
    \begin{array}{ccc}
  x_{1(H_0 +1)} & \ldots & x_{1H} \\
  \vdots & \ddots & \vdots \\
  x_{M(H_0 +1)} & \ldots & x_{MH} \\
    \end{array}
  \right) 
  \left(
   \begin{array}{ccc}
  y_{(H_0 +1)1} & \ldots & y_{(H_0 +1)N} \\
  \vdots & \ddots & \vdots \\
  y_{H1} & \ldots & y_{HN} \\
    \end{array}
  \right)
\right\|^2.
\end{eqnarray*}

Let $\tilde{\lambda}$ be an RLCT of the right-most side, $\lambda_1$ be an RLCT of 1st term in the right-most side, and $\lambda_2$ be an RLCT of the 2nd one. 
Since variables are independent and RLCTs are order isomorphic, 
\begin{eqnarray}
  \label{rlctineq}
  \lambda \leqq \tilde{\lambda} = \lambda_1 + \lambda_2.
\end{eqnarray}
Since the 1st term corresponds to the proof of Lemma \rm{\ref{lemH}} in the case of $H \leftarrow H_0$, 
$$\lambda_1 = H_0 \frac{M+N-1}{2}.$$
In contrast, the 2nd term corresponds to Lemma \rm{\ref{lemH0}} in the case of $H \leftarrow H-H_0$. That causes
$$\lambda_2 = \frac{(H-H_0)\min\{M,N\}}{2}.$$
Using the above two equalities for inequality(\ref{rlctineq}), 
\begin{eqnarray*}
\lambda &\leqq& \tilde{\lambda} \\
&=& \lambda_1 + \lambda_2\\
&=& H_0 \frac{M+N-1}{2} + \frac{(H-H_0)\min\{M,N\}}{2}. 
\end{eqnarray*}
Therefore
$$\lambda \leqq \frac{1}{2} \left[ (H-H_0)\min\{M,N\} +H_0 (M+N-1) \right].$$
\end{proof}

\begin{rem}
Under the same assumption of the Main Theorem, suppose 
\[
  f_{ij}^k = \begin{cases}
    x_{ik}y_{kj} - a_{ik}b_{kj} & (k \in \{1,\ldots,H_0 \}) \\
    x_{ik}y_{kj} & (k \in \{H_0+1,\ldots,H \})
  \end{cases}.
\]
If $f_{ij}^k \geqq 0$, i.e. , $ x_{ik}y_{kj} - a_{ik}b_{kj} \geqq 0 \quad (k \in \{1,\ldots,H_0 \})$, 
$$\lambda = \frac{1}{2} \left[ (H-H_0)\min\{M,N\} +H_0 (M+N-1) \right].$$
\end{rem}

\begin{proof}
Owing to $f_{ij}^k \geqq 0$, 
$$\sum_{k=1}^H (f_{ij}^k)^2 \leqq \left( \sum_{k=1}^H f_{ij}^k \right)^2 \leqq H \sum_{k=1}^H (f_{ij}^k)^2.$$
Thus, 
$$\sum_{k=1}^H (f_{ij}^k)^2 \sim \left( \sum_{k=1}^H f_{ij}^k \right)^2.$$
Using the above relation, 
\begin{eqnarray*}
  && \| XY-AB \|^2  \\
  &=& \sum_{i=1}^M \sum_{j=1}^N \Biggl( \sum_{k=1}^{H_0} (x_{ik}y_{kj} - a_{ik}b_{kj}) 
+ \sum_{k=H_0 +1}^H x_{ik}y_{kj} \Biggr)^2 \\
  &=& \sum_{i=1}^M \sum_{j=1}^N \Biggl( \sum_{k=1}^{H} f_{ij}^k \Biggr)^2 \\
  & \sim & \sum_{i=1}^M \sum_{j=1}^N 
\left( \sum_{k=1}^{H}  (f_{ij}^k)^2 \right) \\
  &=& \sum_{i=1}^M \sum_{j=1}^N 
\left( \sum_{k=1}^{H_0}  (x_{ik}y_{kj} - a_{ik}b_{kj})^2 + \sum_{k=H_0 +1}^H x_{ik}^2 y_{kj}^2 \right) \\
  &=& \sum_{k=1}^{H_0}
\left\|
 x_k y_k^T - a_k b_k^T
\right\|^2 
 + 
\left\|
  \left(
    \begin{array}{ccc}
  x_{1(H_0 +1)} & \ldots & x_{1H} \\
  \vdots & \ddots & \vdots \\
  x_{M(H_0 +1)} & \ldots & x_{MH} \\
    \end{array}
  \right) 
  \left(
   \begin{array}{ccc}
  y_{(H_0 +1)1} & \ldots & y_{(H_0 +1)N} \\
  \vdots & \ddots & \vdots \\
  y_{H1} & \ldots & y_{HN} \\
    \end{array}
  \right)
\right\|^2.
\end{eqnarray*}
Therefore
$$\lambda= \frac{1}{2} \left[ (H-H_0)\min\{M,N\} +H_0 (M+N-1) \right].$$
\end{proof}

\section{Discussion}

In this section, we discuss the result of this paper from four points of view. 

\subsection{Minus Log Bayesian Marginal Likelihood}
First, let us discuss the asymptotic expansion of the minus log Bayesian marginal likelihood, which is defined by
\[
F=-\log\int \varphi(X,Y)\prod_{i=1}^n p(W_i|X,Y)dXdY,
\]
and is sometimes used in model selection and hyperparameter optimization. 
It is well known \cite{Watanabe1} that $F$ has an asymptotic expansion, 
\[
F= -\sum_{i=1}^n\log q(W_i)+\lambda\log n +o_p(\log n),
\]
where $\lambda$ is an RLCT. Therefore, using the results of this paper, we can also derive the upper bound of 
the minus log Bayesian marginal likelihood. 

There is an information criterion called ''singular BIC'' or ''sBIC''. Singular BIC uses the exact value of RLCT or its upper bound in order to approximate the minus log marginal likelihood of singular model \cite{Drton}. NMF is singular statistical model thus our theoretical result can be applied to model selection problems in NMF by using sBIC.

\subsection{Relationship with Reduced Rank Regression}\label{DiscRRR}
Second, the tightness of the upper bound is considered. If the observed $\{W_i\}$ are taken from the 
same probability density function $q(W)$ and if the coefficients of parameter $(X,Y)$ are not limited to 
non-negative values but can be made plus and minus values, then the RLCT is equal to that of the reduced 
rank regression, since the zeta function of NMF is equal to that of reduced rank regression. 
The RLCT of reduced rank regression was clarified in \cite{Aoyagi1}, which is a lower bound of NMF. 
The smallest inner matrix dimension of an NMF is called a nonnegative rank \cite{Cohen}. Suppose $C \in \mathrm{M}(M,N,K)$. Let $\mathrm{rank}_+ C$ be the nonnegative rank of $C$. The following inequality is known \cite{Cohen}:
$$\mathrm{rank} C \leqq \mathrm{rank}_+ C \leqq \min \{ M,N \}.$$
We can also prove that the RLCT of  reduced rank regression is lower than the RLCT of NMF 
in cases where the equality of the Main Theorem is attained, i.e. , in case of Lemma \ref{lemH0} and Lemma \ref{lemH1}. Especially in the case of Lemma \ref{lemH1}, the equal sign holds.
We prove it below.
\begin{proof}
Let $\lambda_{RRR}^r$ and $\lambda_{NMF}^{k}$ be the RLCT of reduced rank regression and the RLCT of NMF in case of Lemma  $k$  where $r$ is the rank of true parameter $AB$. We only have to consider $r=0$ or $r=1$, respectively.
Because of the Main Theorem and Lemma \ref{lemH0} and Lemma \ref{lemH1}, 
\begin{eqnarray*}
\lambda_{NMF}^{\mbox{\ref{lemH0}}} &=& \frac{H\min\{M,N\}}{2}, \\
\lambda_{NMF}^{\mbox{\ref{lemH1}}} &=& \frac{M+N-1}{2}.
\end{eqnarray*}
In contrast, $\lambda_{RRR}^r$ is equal to the following \cite{Aoyagi1} :
\begin{equation}
\label{RRR}
\lambda_{RRR}^r = 
\begin{cases}
\left\{ 2(H+r)(M+N)-(M-N)^2-(H+r)^2 \right\} / 8 & \\
\quad {\rm if}(N+r\leqq M+H \wedge M+r\leqq N+H \wedge H+r\leqq M+N \wedge M+H+N+r:{\rm even}) & \\
\left\{ 2(H+r)(M+N)-(M-N)^2-(H+r)^2 +1 \right\} / 8 & \\
\quad {\rm if}(N+r\leqq M+H \wedge M+r\leqq N+H \wedge H+r\leqq M+N \wedge M+H+N+r:{\rm odd}) & \\
(HM-Hr+Nr) / 2 \quad {\rm if}(M+H<N+r) & \\
(HN-Hr+Mr) / 2 \quad {\rm if}(N+H<M+r) & \\
MN / 2 \quad {\rm if}(M+N<H+r). &
\end{cases}
\end{equation}

\noindent(Case 1) $r=0$ i.e. in the case of Lemma  \ref{lemH0}.

(Case 1-1) $N\leqq M+H \wedge M\leqq N+H \wedge H\leqq M+N \wedge M+H+N:even$.

We assume $N \leqq M$, i.e. , $\lambda_{NMF}^{\mbox{\ref{lemH0}}} = HM/2$. Owing to equality (\ref{RRR}),
\begin{eqnarray*}
\lambda_{RRR}^0 
&=& \left\{ 2H(M+N)-(M-N)^2-H^2 \right\} / 8 \\
&=& \left\{ 2(MH+HN+NM)-M^2-N^2-H^2 \right\} / 8. \\
 \lambda_{RRR}^0-\lambda_{NMF}^{\mbox{\ref{lemH0}}}
&=& -\left\{ M^2+H^2+N^2-2(MH+HN+NM)-4HM \right\} / 8 \\
&=& -\left\{ M^2+H^2+N^2-2(-MH+HN+NM) \right\} / 8 \\
&=& -\left[ (-M)^2+(-H)^2+N^2+2\{(-M)(-H)+(-H)N+N(-M)\} \right] / 8 \\
&=& -(N-M-H)^2 /8 \leqq 0
\end{eqnarray*}
Therefore
$$ \lambda_{RRR}^0 \leqq \lambda_{NMF}^{\mbox{\ref{lemH0}}}.$$
If $N>M$, that can be derived in the same way as above.

(Case 1-2) $N\leqq M+H \wedge M\leqq N+H \wedge H\leqq M+N \wedge M+H+N:odd$.

We assume $N \leqq M$, i.e. , $\lambda_{NMF}^{\mbox{\ref{lemH0}}} = HM/2$.
In the same way as Case 1-1,
\begin{eqnarray*}
\lambda_{RRR}^0-\lambda_{NMF}^{\mbox{\ref{lemH0}}}
&=&1/8 -(N-M-H)^2 /8 \\ 
&=&1^2/8 -(M+H-N)^2 /8 \\
&=& -(M+H-N+1)(M+H-N-1) /8.
\end{eqnarray*}
If $H=0$, $N+1 \leqq M$ is attained because of that $M+N$ must be odd and $N \leqq M$, hence $\lambda_{RRR}^0 \leqq \lambda_{NMF}^{\mbox{\ref{lemH0}}}$.
Else, owing to $H \geqq 1$ and $N \leqq M$, 
\begin{eqnarray*}
N+1 &\leqq& M+1 \\
&\leqq& M+H.
\end{eqnarray*}
We also get $N-1 \leqq N+1 \leqq M+H$ thus
$$-(M+H-N+1)(M+H-N-1) \leqq 0.$$
Therefore, $$\lambda_{RRR}^0 \leqq \lambda_{NMF}^{\mbox{\ref{lemH0}}}.$$
If $N>M$, that can be derived in the same way as above.

(Case 1-3) $N+H<M$ i.e. $N<N+H<M$.

On account of $N<M$,$\lambda_{NMF}^{\mbox{\ref{lemH0}}}=HM/2=\lambda_{RRR}^0$.

(Case 1-4) $M+H<N$ i.e. $M<M+H<N$.

In the same way as Case 1-3, $\lambda_{NMF}^{\mbox{\ref{lemH0}}}=HN/2=\lambda_{RRR}^0$.

(Case 1-5) $M+N<H$ i.e. $N < M+N <H \wedge M < M+N <H$.

On account of $M<H \wedge N<H$, $MN<HN \wedge MN<HM$ i.e. 
$MN < H \min\{M,N \}$.
Thus 
$$\lambda_{NMF}^{\mbox{\ref{lemH0}}}=H \min\{M,N \}/2>MN/2=\lambda_{RRR}^0 .$$

From the above, $\lambda_{RRR}^0 \leqq \lambda_{NMF}^{\mbox{\ref{lemH0}}}$.

\noindent(Case 2) $r=1 \wedge H=1$ i.e. in case of Lemma \ref{lemH1}.

As for the claim $\lambda_{RRR}^1 \leqq \lambda_{NMF}^{\mbox{\ref{lemH1}}}$, 
we only have to prove in the case of $M+H<N+r$ and $N+H<M+r$ since it contradicts in other cases. 
In the first and second cases, owing to $r=H=1$, it becomes $M<N \wedge N<M$ and contradicts. In the last(fifth) case, it is $2 \leqq M+N<2$ and conflicts. Thus we consider the third and fourth cases:$M+H<N+r$ and $N+H<M+r$, respectively.

(Case 2-1) $M+H<N+r$ i.e. $M<N$.

According to the equality (\ref{RRR}) and Lemma  \ref{lemH1}, 
\[
\lambda_{RRR}^1 
= \frac{M+N-1}{2} 
= \lambda_{NMF}^{\mbox{\ref{lemH1}}}.
\]

(Case 2-2) $N+H<M+r$, i.e. , $N<M$.

That can be derived in the same way as Case 2-1.

From the above, $\lambda_{RRR}^1 = \lambda_{NMF}^{\mbox{\ref{lemH1}}}$.
\end{proof}
That is why the RLCT of the reduced rank regression is considered as a lower bound of the RLCT of the NMF.

$H_0$ and $r$ are denoted by the non-negative rank and the usual rank respectively. We compare the RLCT of reduced rank regression with the exact value and upper bound and summarize in Table\ref{comparison} below in Table section. Note that there is no non-negative matrix which has larger rank than non-negative rank. Moreover, if $M \leqq 3$ or $N \leqq 3$, then $\mathrm{rank} C = \mathrm{rank}_+ C$, where $C \in \mathrm{M}(M,N,K)$\cite{Cohen}. For example, the following $4 \times 4$ non-negative matrix $C$ satisfying
 $\mathrm{rank} C=3$ and $\mathrm{rank}_+ C=4$ is well-known \cite{Cohen}:
\[
C=\left(
    \begin{array}{cccc}
	1 & 1 & 0 & 0 \\
	1 & 0 & 1 & 0 \\
	0 & 1 & 0 & 1 \\
	0 & 0 & 1 & 1 \\
    \end{array}
  \right).
\]

\subsection{Robustness on Distribution}
Third, we study the generalization of the result to another distribution. 
In Theorem\ref{thm:bayes}, we studied a case where
the matrix $W$ is generated by normal distributions whose averages are  $AB$ and $XY$.
Then Kullback-Leibler divergence $KL(X,Y)$ satisfies $KL(X,Y) \sim \| XY-AB\|^2$, as is well known \cite{Aoyagi1}. 
If $W$ is generated by Poisson distributions, due to the conditions of the parameters of these distributions, elements of $XY$ are restricted by positive elements. However, the Kullback-Leibler divergence has same RLCT as the square error if elements of $XY$ are positive. 
\begin{prop}\label{poissonNMF}
Let the probability density functions of $W\in \mathrm{M}(M,N,K)$ be $q(W)$ and $p(W|X,Y)$, which represent 
a true distribution and a learning machine respectively defined by 
\begin{eqnarray*}
q(W) & \propto & \pois (W|AB), \\
p(W|X,Y) & \propto & \pois (W|XY),
\end{eqnarray*}
where $\pois(W|C)$ is a probability density function of the Poisson distribution with average C.
Also let $\varphi(X,Y) $ be a probability density function such that it is positive on a compact subset of 
$\mathrm{M}(M,H,K_0)$ $\times$ $\mathrm{M}(H,N,K_0)$.
Then, the Kullback-Leibler divergence has the same RLCT as the square error.
\end{prop}
\begin{proof}
Let $x \in \mathbb{N} \cup \{0\}$, $a>0$, $b>0$. 
We put 
\begin{eqnarray*}
  p(x|a) &:=& \frac{e^{-a} a^x}{x!}, \\
  K(a,b) &:=& \sum^{\infty}_{x=0}p(x|a) \mathrm{log} \frac{p(x|a)}{p(x|b)}.
\end{eqnarray*}
Using
\begin{eqnarray*}
  \mathrm{log} \frac{p(x|a)}{p(x|b)} &=& \mathrm{log} \frac{e^{-a} a^x}{x!} - \mathrm{log} \frac{e^{-b} b^x}{x!} \\
  &=& -a + x \mathrm{log} a +b - x \mathrm{log} b \\
  &=& b-a + x \mathrm{log} \frac{a}{b}
\end{eqnarray*}
and
$$\sum^{\infty}_{x=0}p(x|a)=1,$$
\begin{eqnarray}
  K(a,b) &=& b-a + \mathrm{log} \frac{a}{b} \sum^{\infty}_{x=1} a \frac{e^{-a} a^{x-1}}{(x-1)!} \nonumber \\
  &=& b-a + a \mathrm{log} \frac{a}{b}. \label{Idiv}
\end{eqnarray}
Then, owing to
\begin{eqnarray*}
\partial_a K(a,b) &=& \mathrm{log}a - \mathrm{log}b, \\
\partial_b K(a,b) &=& 1-a/b,
\end{eqnarray*}
and that a log function is monotone increasing, 
$$\partial_a K(a,b)=\partial_b K(a,b)=0 \Leftrightarrow a=b.$$
Signs of the above partial derivations are
$$\partial_a K(a,b)>0 \wedge \partial_b K(a,b)<0 \, \rm{in \ case \ of \ }a>b$$
$$\partial_a K(a,b)<0 \wedge \partial_b K(a,b)>0 \, \rm{in \ case \ of \ }a<b$$

On account of that and smoothness, the increase or decrease and convexity of $K(a,b)$ is the same as those of $(b-a)^2$.
Hence $\exists c_1 , c_2>0$ s.t. 
\begin{eqnarray}
\label{KL2RSS}
c_1 (b-a)^2 \leqq K(a,b) \leqq c_2 (b-a)^2.
\end{eqnarray}
i.e. $K(a,b) \sim (b-a)^2$.

We assumed that elements of matrices are generated by Poisson distributions. Using inequality(\ref{KL2RSS}) for each element, 
$$\mbox{Kullback-Leibler divergence} \sim \|XY-AB\|^2$$
,where $a$ is an element of $AB$ and $b$ is an element of $XY$.
\end{proof}

We also consider the case which $W$ is generated by exponential distribution. The elements of $XY$ must be restricted
by positive ones as in the case of Poisson. However, also in this case, the Kullback-Leibler divergence has same RLCT as the square error if elements of $XY$ are positive. 

\begin{prop}\label{expoNMF}
Let the probability density functions of $W\in \mathrm{M}(M,N,K)$ be $q(W)$ and $p(W|X,Y)$, which represent 
a true distribution and a learning machine respectively defined by 
\begin{eqnarray*}
q(W) & \propto & \expo (W|AB), \\
p(W|X,Y) & \propto & \expo (W|XY),
\end{eqnarray*}
where $\expo(W|C)$ is a probability density function of  the exponential distribution with average C.
Also let $\varphi(X,Y) $ be a probability density function such that it is positive on a compact subset of 
$\mathrm{M}(M,H,K_0)$ $\times$ $\mathrm{M}(H,N,K_0)$.
Then, the Kullback-Leibler divergence has same RLCT as the square error.
\end{prop}

\begin{proof}
Let $x >0$, $a>0$, $b>0$. 
We put 
\begin{eqnarray*}
  p(x|a) &:=& \frac{e^{-x/a}}{a}, \\
  K(a,b) &:=& \int p(x|a) \mathrm{log} \frac{p(x|a)}{p(x|b)} \d x.
\end{eqnarray*}
Using
\begin{eqnarray*}
\mathrm{log} \frac{p(x|a)}{p(x|b)} &=& \log \frac{b}{a} e^{-x/a+x/b} \\
&=& \log b -\log a -\frac{x}{a} + \frac{x}{b},
\end{eqnarray*}
$$\int p(x|a) \d x = 1$$
and
$$\int x p(x|a) \d x = \mathbb{E}[x]=a,$$
\begin{eqnarray}
K(a,b)&=&\int p(x|a) \left( \log b - \log a -\frac{x}{a} + \frac{x}{b} \right)  \d x \nonumber \\
&=& \log b -\log a -1 + \frac{a}{b}. \label{ISdiv}
\end{eqnarray}
We immediately get
\begin{eqnarray*}
\partial_a K(a,b) &=& \frac{1}{b}-\frac{1}{a}, \\
\partial_b K(a,b) &=& \frac{1}{b}-\frac{a}{b^2} \\
&=& \frac{b-a}{b^2}
\end{eqnarray*}
and increase or decrease of $K$.
Thus this proposition can be proved in the same way as Proposition \ref{poissonNMF} as follows. 
\end{proof}

\begin{rem}
The right side of equality$(\ref{Idiv})$ and $(\ref{ISdiv})$ are respectively equal to {\rm ''I-divergence''}{\rm \cite{Tusnady}} and {\rm ''Itakura-Saito-divergence''}{\rm \cite{Itakura}} which are used as criterion of difference between observed matrix and reproduced matrix in NMF{\rm \cite{Lee,Fevotte,Finesso}}.
\end{rem}

As a result, we can apply the Main Theorem to Bayesian learning if we use a Kullback-Leibler divergence or square error as a criterion of difference between $XY$ and $AB$ in cases where elements of matrices are generated by normal, Poisson or exponential distributions. Thus, the upper bound of the average of generalization errors can be clarified if the size and inner dimension of the observed matrix and reproduced matrix. 

\subsection{Future Study}
Lastly, we discuss future study in this subsection. There are two works.

The first is about tightness of our upper bound.
The upper bound derived in this paper is the first theoretical result about Bayesian NMF. 
Hence there is not yet clarified its upper bound which is lower than our upper bound. However, in this paper, we also 
showed that there are cases when the exact values are obtained as Lemma \ref{lemH0} and Lemma \ref{lemH1}. 
Based on the exact cases, it is expected that the upper bound is tight, however, its mathematical proof is the future study. 

The second is about numerical experiments. In NMF, even in non-Bayesian cases\cite{Paatero}, the accurate simulation 
method is not yet established. The numerical calculation of them strongly depends on the initial 
values. It seldom goes to the minimum point because it has many local minima. 
Also in Bayesian cases\cite{Cemgil}, the effective Markov Chain Monte Carlo method ( MCMC ) is not yet devised 
for NMF, where MCMC is a method to construct posterior in numerical Bayesian estimation. 
Thus we have no method by which can be used for verifying the upper bound. 
Our theorem is the first theoretical base for construction of numerical calculation. It is the 
future study to make the numerical method for Bayesian NMF based on our theoretical result.

\section{Conclusion}
The upper bound of a real log canonical threshold of the non-negative matrix factorization is 
derived and an application to Bayesian learning is introduced. 
Future work is to clarify the tightness of the upper bound mathematically and make numerical method to consider numerical verification in Bayesian non-negative matrix factorization. 

\section*{Acknowledgments}

This research was partially supported by the Ministry of Education, Science, Sports and Culture in
Japan, Grant-in-Aid for Scientific Research 15K00331.

\section*{Table}
In this section, we show the table mentioned in Discussion \ref{DiscRRR}. ''-'' means that there is no matrix satisfying the condition written in the table. ''( {\it value} )'' is denoted that this {\it value} is not lower bound of the RLCT of NMF ( i.e. our result ) since there is no non-negative matrix satisfying conditions such as $H_0=4$ and $M=N=3$.
\begin{table}[htb]
\caption{Comparison the RLCT of NMF with that of reduced rank regression}
  \begin{tabular}{|c|c|c|c|c|c|} \hline
   \multicolumn{2}{|l|}{model size and model / matrix size}
      & \multicolumn{1}{c|}{M=N=2} & \multicolumn{1}{c|}{M=N=3} & \multicolumn{1}{c|}{M=N=4} & \multicolumn{1}{c|}{M=N=5} \\ \hline 
	$H=M$,$H_0=0$ & NMF ( exact value ) & \bf{2} & \bf{9/2} & \bf{8} & \bf{25/2} \\ \cline{2-6}
	$r=0$ & reduced rank regression & 3/2 & 7/2 & 6 & 19/2 \\ \hline
	$H=H_0=1$ & NMF ( exact value ) & \bf{3/2} & \bf{5/2} & \bf{7/2} & \bf{9/2} \\ \cline{2-6}
	$r=1$ & reduced rank regression & 3/2 & 5/2 & 7/2 & 9/2 \\ \hline
	$H=H_0=2$ & NMF ( bound ) & \bf{3} & \bf{5} & \bf{7} & \bf{9} \\ \cline{2-6}
	$r=2$ & reduced rank regression & 2 & 4 & 6 & 8 \\ \hline
	$H=H_0=3$ & NMF ( bound ) & \bf{-} & \bf{15/2} & \bf{21/2} & \bf{27/2} \\ \cline{2-6}
	$r=3$ & reduced rank regression & - & 9/2 & 15/2 & 21/2 \\ \hline
	$H=H_0=4$ & NMF ( bound ) & \bf{-} & \bf{-} & \bf{14} & \bf{18} \\ \cline{2-6}
	$r=3$ & reduced rank regression & - & (9/2) & 8 & 23/2 \\ \cline{3-6}
	$r=4$ & \, & - & - & 8 & 12 \\ \hline
	$H=H_0=5$ & NMF ( bound ) & \bf{-} & \bf{-} & \bf{-} & \bf{45/2} \\ \cline{2-6}
	$r=3$ & \, & - & (9/2) & (8) & 12 \\ \cline{3-6}
	$r=4$ & reduced rank regression & - & - & (8) & 25/2 \\ \cline{3-6}
	$r=5$ & \, & - & - & - & 25/2 \\ \hline
  \end{tabular}\label{comparison}
\end{table}

\end{document}